\documentclass[12pt]{amsart}
\usepackage[all]{xy}

\usepackage{fullpage}
\usepackage{setspace}
\usepackage{amsmath,amsthm,amssymb, paralist, xspace, graphicx, url, amscd, euscript, mathrsfs,epic,eepic,verbatim,stmaryrd}
\usepackage{caption}

\usepackage[usenames,dvipsnames]{color}

\usepackage[colorlinks=true]{hyperref}

\numberwithin{equation}{section}

1

\numberwithin{subsection}{section}

\allowdisplaybreaks[1]

\newtheorem*{namedtheorem}{\theoremname}
\newcommand{\theoremname}{testing}

\newtheorem{theorem}{Theorem}[section]
\newtheorem{proposition}[theorem]{Proposition}
\newtheorem{proposition-definition}[theorem]
{Proposition-Definition}
\newtheorem{corollary}[theorem]{Corollary}
\newtheorem{lemma}[theorem]{Lemma}
\newtheorem{conjecture}[theorem]{Conjecture}

\theoremstyle{definition}
\newtheorem{definition}[theorem]{Definition}
\newtheorem{remark}[theorem]{Remark}

\renewcommand{\thesubsubsection}{\ifnum\value{subsection}=0
	\arabic{section}.\arabic{subsubsection}%
\else
	\arabic{section}.\arabic{subsection}.\arabic{subsubsection}%
\fi}

\makeatletter
\@addtoreset{subsubsection}{section}

\let\c@equation\c@subsubsection

\let\subsection@old\subsection
\def\subsection#1{\ifnum\value{subsubsection}>0 \ifnum\value{subsection}=0
	\setcounter{subsection}{\value{subsubsection}}%
\fi \fi
\subsection@old{#1}}
\makeatother

\theoremstyle{remark}

\newcommand{\cX}{{\mathcal X}} 

\newcommand{\sD}{\mathscr{D}}

\newcommand{\sZ}{\mathscr{Z}}
\newcommand{\sX}{\mathscr{X}}

\newcommand{\defi}[1]{\textsf{#1}} 

\newcommand\cA{\mathcal{A}}

\newcommand\cD{\mathcal{D}}

\newcommand\cO{\mathcal{O}}

\newcommand{\oX}{{\overline{X}}}

\newcommand{\ocA}{{\overline{\mathcal{A}}}}
\newcommand{\tcA}{{\widetilde{\mathcal{A}}}}
\newcommand{\otcA}{{\overline{\widetilde{\mathcal{A}}}}}

\newcommand\QQ{\mathbb{Q}}

\newcommand\ZZ{\mathbb{Z}}

\newcommand\Spec{\operatorname{Spec}}

\begin{document}

\title[Campana, Vojta, and level structures]{Campana points, Vojta's conjecture, and\\
level structures on semistable abelian varieties
}

\author[D. Abramovich]{Dan Abramovich}
\address{Department of Mathematics, Box 1917, Brown University, Providence, RI, 02912, U.S.A}
\email{abrmovic@math.brown.edu}

\author[A. V\'arilly-Alvarado]{Anthony V\'arilly-Alvarado}
\address{Department of Mathematics MS 136, Rice University, 6100 S.\ Main St., Houston, TX 77005, USA}
\email{av15@rice.edu}


\thanks{Research by D. A. partially supported by NSF grant DMS-1500525. Research by A. V.-A. partially supported by NSF CAREER grant DMS-1352291. This continues work of the authors that began during the workshop ``Explicit methods for modularity of K3 surfaces and other higher weight motives'', held at ICERM in October, 2015. We thank the organizers of the workshop and the staff at ICERM for creating the conditions that sparked this project. We thank {\sc David Zureick-Brown} for comments and {\sc Keerthi Padapusi Pera} for the appendix of \cite{AV-Campana-Vojta}.}

\date{\today}
\maketitle
\setcounter{tocdepth}{1}


\section{Introduction}
\label{s:intro}


Fix a number field $K$ with ring of integers $\cO_K$, as well as a finite set $S$ of places of $K$ that contains the archimedean places.  Denote by $\cO_{K,S}$ the corresponding ring of $S$-integers.
The purpose of this short note is to introduce a qualitative conjecture, in the spirit of {\sc Campana}, to the effect that certain subsets of rational points on a variety over $K$ or a Deligne--Mumford stack over $\cO_{K,S}$ cannot be Zariski dense; see Conjecture~\ref{conj:Campana}.  This conjecture interpolates, in a way that we make precise, between {\sc Lang}'s conjecture for rational points on varieties over $K$ of general type, and the conjecture of {\sc Lang} and {\sc Vojta} that asserts that $\cO_{K,S}$-points on a variety of logarithmic general type are not Zariski-dense.  One might thus expect our conjecture to follow from {\sc Vojta}'s quantitative conjecture on integral points; we show this is the case.  As an application we show, assuming Conjecture~\ref{conj:Campana}, that for a fixed positive integer $g$, there is an integer $m_0$ such that, for any $m > m_0$, no principally polarized abelian variety $A/K$ of dimension $g$ with semistable reduction outside of $S$ has full level-$m$ structure.

Let $\sX \to \Spec \cO_{K,S}$ be a smooth proper morphism from a scheme or Deligne--Mumford stack, and let $\sD$ be a fiber-wise normal crossings divisor on $\sX$. We say that $(\sX,\sD)$ is a \defi{normal crossings model} of its generic fiber $(X,D)$. Write $\sD = \sum_i \sD_i$ and let $D_i$ be the generic fiber of $\sD_i$. For each $D_i$ appearing in $D$, choose $0\leq \epsilon_i \leq 1$; write $\vec\epsilon = (\epsilon_1,\epsilon_2,\dots)$, $D_{\vec\epsilon} = \sum_i\epsilon_iD_i$, and $\sD_{\vec\epsilon} = \sum_i\epsilon_i\sD_i$.  Given a maximal ideal $q$ of $\cO_{K,S}$ with localization $\cO_{K,q}$, and a point $x \in \sX(\cO_{K,q})$, we denote by $n_q(\sD_i,x)$ the intersection multiplicity of $x$ and $\sD_i$, and for real numbers $a_i$, we define $n_q(\sum_i a_i \sD_i,x)$ as $\sum_i a_in_q(\sD_i,x)$; see \S\ref{s:IntMults} for details.

\begin{definition}
	\label{def:Campana}
	A point $x \in \sX(\cO_{K,S})$ is called an \defi{$\vec\epsilon$-Campana point} of $(\sX,\sD)$ if for every maximal ideal $q\subset\cO_{K,S}$ such that $n_q(\sD,x) > 0$ we have $n_q(\sD_{\vec\epsilon},x) \geq 1$.  We write $\sX(\cO_{K,S})_{\sD_{\vec\epsilon}}$ for the set of $\vec\epsilon$-Campana points of $(\sX,\sD)$.
\end{definition}

\begin{conjecture}[$\vec\epsilon$-Campana  Conjecture]
	\label{conj:Campana}
	If $K_X + \sum(1 - \epsilon_i)D_i$ is big, then $\sX(\cO_{K,S})_{\sD_{\vec\epsilon}}$ is not Zariski-dense in $X$.
\end{conjecture}

\begin{remark}
{\sc Campana} stated the case of Conjecture~\ref{conj:Campana} for curves in~\cite[Conjecture~4.5 and Remark~4.6]{Campana-surfaces}. {\sc Abramovich} gave a higher dimensional statement of it in~\cite[Conjecture~2.4.19]{A-Clay}. Conjecture~\ref{conj:Campana} is a streamlined version of this generalization.
\end{remark}

\begin{remark}
	Setting $\epsilon_i = 1$ for all $i$ in Conjecture~\ref{conj:Campana}, the condition $n_q(\sD_{\vec\epsilon},x) \geq 1$ is automatically satisfied for all rational points and all $q$. Also  $K_X + \sum(1 - \epsilon_i)D_i = K_X$, hence we recover {\sc Lang}'s conjecture for rational points on varieties of general type. 
	 At the other end of the spectrum, setting all the $\epsilon_i = 0$,  the condition $n_q(\sD_{\vec\epsilon},x) \geq 1$ can only be satisfied if $n_q(\sD,x) = 0$ at all $q\subset\cO_{K,S}$, so $x$ is $S$-integral on $\cX \smallsetminus \cD$. Hence we get the {\sc Lang--Vojta} conjecture: $S$-integral points on a variety of logarithmic general type are not Zariski-dense. 
	 We show in \S\ref{s:Vojta} that Conjecture~\ref{conj:Campana} follows from {\sc Vojta}'s conjecture.
\end{remark}


\subsection{Application}
We give an application of Conjecture~\ref{conj:Campana}, in the spirit of our recent work~\cite{Alevels,AV-Campana-Vojta}. Recall that, for a positive integer $m$, a \defi{full level}-$m$ structure on an abelian variety $A/K$ of dimension $g$ is an isomorphism of group schemes on the $m$-torsion subgroup
	\begin{equation}
		\label{eq:levelstructure}
		\phi\colon A[m] \,\xrightarrow{\ \sim\ }\, (\ZZ/m\ZZ)^g \times (\mu_{m})^g.
	\end{equation}
\begin{theorem}
	\label{thm:main}
	Let $K$ be a number field, $S$ a finite set of places, and let $g$ be a positive integer. Assume Conjecture \ref{conj:Campana}. Then there is an integer $m_0$ such that, for any $m > m_0$, no principally polarized abelian variety $A/K$ of dimension $g$ with semistable reduction outside $S$ has full level-$m$ structure.
\end{theorem}

\begin{remark}
In~\cite{AV-Campana-Vojta}, we prove a version of Theorem~\ref{thm:main} without a semistability assumption on the abelian variety $A$, at the cost of assuming {\sc Vojta}'s conjecture.
\end{remark}

The idea behind Theorem~\ref{thm:main} is the following.  Let $(\tcA_g)_K$ denote the moduli stack parametrizing principally polarized abelian varieties of dimension $g$ over $K$. We show in Proposition~\ref{prop:inCampana} that for $\epsilon > 0$ and $\vec\epsilon = (\epsilon,\epsilon,\dots)$, if $X \subseteq (\tcA_g)_K$ is closed, then the set $X(K,S)_{\geq m_0}$ of $K$-rational points of $X$ corresponding to abelian varieties $A/K$ admitting full level-$m$ structure for some $m \geq m_0$ and having semistable reduction outside of $S$ lies inside an $\vec\epsilon$-Campana set, for $m_0 \gg 0$. We then use a result on logarithmic hyperbolicity~\cite[Theorem~1.6]{Alevels} to verify the hypothesis of Conjecture~\ref{conj:Campana} in this case, and a Noetherian induction argument to conclude the proof of Theorem~\ref{thm:main}.




\section{Proof of Theorem \ref{thm:main}}


\subsection{Moduli spaces and toroidal compactifications}

	We follow the notation of~\cite[\S4]{AV-Campana-Vojta}, working over $\Spec \ZZ$:
	
	\medskip

	\begin{tabular}{ll}
		$\tcA_g \subset \otcA_g$ & a toroidal compactification of the moduli \emph{stack} of \\
	& principally polarized abelian varieties of dimension $g$\\[1mm]
		$\cA_g \subset \ocA_g$ & the resulting compactification of the moduli \emph{space} of \\
	& principally polarized abelian varieties of dimension $g$\\[1mm]
	\end{tabular}
	
	\begin{tabular}{ll}
		$\tcA_g^{[m]} \subset \otcA_g^{[m]}$ & a compatible toroidal compactification of the moduli \emph{stack} of \\
	& principally polarized abelian varieties of dimension $g$\\	
	& with full level-$m$ structure\\[1mm]
		$\cA_g^{[m]} \subset \ocA_g^{[m]}$ & the resulting compactification of the moduli space of \\
	& principally polarized abelian varieties of dimension $g$\\
	& with full level-$m$ structure\\[1mm]
	\end{tabular}
	
	\medskip
	As noted in~\cite{AV-Campana-Vojta}, we may use a construction by {\sc Faltings} and {\sc Chai}~\cite{Faltings-Chai} of the stack $\otcA_g^{[m]}$, which is a priori smooth over $\Spec \ZZ[1/m,\zeta_m]$, where $\zeta_m$ is a primitive $m$-th root of unity, to obtain a stack we denote by $(\otcA_g^{[m]})_{\ZZ[1/m]}$, smooth over $\ZZ[1/m]$. This stack is extended over all of $\Spec \ZZ$ by defining $\otcA_g^{[m]}$ as the normalization of $\otcA_g$ in  $(\otcA_g^{[m]})_{\ZZ[1/m]}$. Unfortunately, even the interior of this stack over primes dividing $m$ does not have a modular interpretation. See, however, the results of {\sc Madapusi Pera} in \cite[Appendix A]{AV-Campana-Vojta}.
	
\subsection{Semistability and integrality} We require the following well-known statement essentially contained in \cite{Faltings-Chai}.

\begin{proposition}
\label{Prop:integrality} 
Let $K$ be a number field, $S$ a finite set of places. Let $A/K$ be a principally polarized abelian variety with full level-$m$ structure, and with semistable reduction outside of $S$. Then the point $x_m\colon\Spec K \to \tcA_g^{[m]}$ associated to $A$ extends to an integral point $\xi_m\colon\Spec \cO_{K,S} \to  \otcA_g^{[m]}$.
\end{proposition}

\begin{proof} First consider the case $m=1$. By \cite[Theorem IV.5.7(5)]{Faltings-Chai} the extension exists if and only if, for every prime $q\not\in S$ and for any strictly henselization $V$ of $\cO_{K,q}$ with valuation $v$, the bimultiplicative form $v\circ b$ corresponds to a point of  a cone only depending on $q$. (Here $b$ is the symmetric  bimultiplicative form associated to the degeneration of  $A$ at $q$ by the theory of degenerations, as indicated in \cite[Proposition IV.5.1]{Faltings-Chai}.) This condition is automatic for a Dedekind domain such as $\cO_{K,q}$, see \cite[Remark IV.5.3]{Faltings-Chai}, hence our proposition holds in case $m=1$.

To prove the statement in general, consider the point $x\colon \Spec K \to \tcA_g$ obtained by composing $x_m$ with $\tcA_g^{[m]} \to \tcA_g$. Since the proposition holds for $m=1$, the point $x$ extends to $\xi\colon \Spec \cO_{K,S} \to  \otcA_g$. Since $\otcA_g^{[m]} \to \otcA_g$ is representable and finite, the stack $Z:=\Spec \cO_{K,S} \times_{\otcA_g} \otcA_g^{[m]}$, where the projection on the left is $\xi$,  is in fact a scheme finite over $\Spec \cO_{K,S}$. The point $x_m$ defines a point $\Spec K \to Z$, which extends to a point $\Spec \cO_{K,S} \to Z$ by the valuative criterion for properness. Composing with the projection $Z \to \otcA_g^{[m]}$ gives the desired point $\xi_m$. 
\end{proof}


\subsection{Intersection multiplicities}
\label{s:IntMults}

Let $(\sX,\sD)$ be a normal crossings model, and let $I_{\sD_i}$ denote the ideal of $\sD_i$. Given a maximal ideal $q$  of $\cO_{K,S}$ with localization $\cO_{K,q}$, and a point $x \in \sX(\cO_{K,q})$, define $n_q(\sD_i,x)$ through the equality of ideals in $\cO_{K,q}$
\[
	I_{\sD_i}\big|_x = q^{n_q(\sD_i,x)}.
\]
We call $n_q(\sD_i,x)$ the intersection multiplicity of $x$ and $\sD_i$.


\subsection{Notation for substacks}

	Let $X \subseteq (\tcA_g)_K$ be a closed substack, let $X'\to X$ be a resolution of singularities, $X' \subset \overline X'$ a smooth compactification with $D = \overline X' \smallsetminus X'$ a normal crossings divisor. Assume that the rational map $f\colon\oX' \to \otcA_g$ is a morphism. Let $X'_m = X'\times_{\tcA_g}\tcA_g^{[m]}$, and let $\oX'_m \to\oX'\times_{\otcA_g} \ocA_g^{[m]}$ be a resolution of singularities with projections $\pi_m^X\colon \oX'_m\to \overline{X}'$ and $f_m\colon \oX'_m\to \otcA_g^{[m]}$.

	We now spread these objects over $\cO_{K,S}$ for a suitable finite set of places $S$ containing the archimedean places. Let $(\sX,\sD)$ be a normal crossings model of $(\oX',D)$ over $\Spec \cO_{K,S}$. As above, write $\sD = \sum_i \sD_i$. Such a model exists, even for Deligne--Mumford stacks, by \cite[Proposition~2.2]{Olsson}.  To avoid clutter, in the special case when $\vec\epsilon = (\epsilon,\epsilon,\dots)$, an $\vec\epsilon$-Campana point of $\sX$ shall be called an $\epsilon$-Campana point, and we write $\sX(\cO_{K,S})_{\epsilon\sD}$ for the set of $\epsilon$-Campana points of $(\sX,\sD)$.


\subsection{Levels and Campana points}

	Let $X(K,S)_{[m]}$ be the set of $K$-rational points of $X$ corresponding to principally polarized abelian varieties $A/K$ with semistable reduction outside $S$, admitting full level-$m$ structure. Define
	\begin{equation}
	\label{eq:toinfty}
		X(K,S)_{\geq m_0} := \bigcup_{m \geq m_0} X(K,S)_{[m]}.
	\end{equation}

\begin{proposition}
	\label{prop:inCampana}
	Fix $\epsilon > 0$. Then there exists $m_0$ such that $X(K,S)_{\geq m_0}$ is contained in the set $\sX(\cO_{K,S})_{\epsilon\sD}$ of $\epsilon$-Campana points of $(\sX,\sD)$.
\end{proposition} 

\begin{proof}
  Let $x_m \in X'_m(K)$, and write $\pi_m(x_m) =: x$ for its image in $\oX'(K)$. Let $q \notin S$ be a finite place of $K$, and let $\cO_{K,q}$ be the corresponding local ring. Let $\xi\colon \Spec\cO_{K,q} \to \oX'$ and $\xi_m\colon \Spec\cO_{K,q} \to \oX'_m$ be the extensions of $x$ and $x_m$ to $\Spec \cO_{K,q}$, which exist by Proposition \ref{Prop:integrality}.
  	
	Write $E$ for the boundary divisors of $\left(\otcA_g\right)_K$; on $\oX'$ we have an equality of divisors
	\[
		f^*E = \sum a_iD_i,
	\]
where each $a_i>0$; see \cite[Equation (4.3)]{Alevels}.  It follows from~\cite[Proposition~4.3]{AV-Campana-Vojta} that there exists an integer $M$ depending only on $g$ such that
\[
	n_q(\sD,x) \geq \frac{m}{M\cdot\max\{a_i\}}.
\]
We note that in our case we can take $M = 1$ because $x$ and $x_m$ could be extended to $\cO_{K,q}$-points, as the proof of~\cite[Proposition~4.3]{AV-Campana-Vojta} shows. Thus, if $m \geq \max\{a_i\}/\epsilon$, the point $x \in \oX'(K)$ is an $\epsilon$-Campana point.
\end{proof}

\begin{proof}[Proof of Theorem~\ref{thm:main}]
	We proceed by Noetherian induction. For each integer $i \geq 1$, let
	\[
		W_i = \overline{\tcA_g(K)_{\geq i}}.
	\]
	Note that $W_i$ is a closed subset of $\cA_g$, and that $W_i \supseteq W_{i+1}$ for every $i$. The chain of $W_i$ must stabilize by the Noetherian property of the Zariski topology of $\cA_g$.  Say $W_n = W_{n+1} = \cdots$. 

	We claim that $W_n$ has dimension $\leq 0$. Suppose not, and let $X \subseteq W_n$ be an irreducible component of positive dimension. Fix $\epsilon > 0$ so that $K_X + (1 - \epsilon)D$ is big: such an $\epsilon$ exists by~\cite[Corollary~1.7]{Alevels}. Hence the hypothesis of Conjecture~\ref{conj:Campana} holds (with all $\epsilon_i$ equal to $\epsilon$.) By Conjecture~\ref{conj:Campana}, the set $\sX(\cO_{K,S})_{\epsilon\sD}$ of $\epsilon$-Campana points is not Zariski dense in $X$. On the other hand, Proposition~\ref{prop:inCampana} shows there is an integer $m_0$ such that $X(K,S)_{\geq m_0} \subseteq \sX(\cO_{K,S})_{\epsilon\sD}$, from which is follows that $X(K,S)_{\geq m_0}$ is not Zariski-dense in $X$. Thus $W_{m_0}$, which equals $W_n$, does not contain $X$, and thus $X$ is not an irreducible component after all. This proves that $\dim W_n \leq 0$.
		
	Finally, if $W_n$ is a finite set of points, then we can apply the {\sc Mordell--Weil} theorem to conclude that $W_n(K)_{[m]} = \emptyset$ for all $m \gg 0$.
\end{proof}


\section{Vojta's conjecture and Campana points}
\label{s:Vojta}
\subsection{Counting functions for integral points}

Let $(\sX,\sD)$ be a normal crossings model. For $q \in \cO_{K,S}$, we denote by $\kappa(q)$ the residue field of the associated local ring. Following {\sc Vojta}~\cite[p.~1106]{VojtaABC}, for $x \in \sX(\cO_{K,S})$,
 define the \defi{counting function}
\begin{equation}
\label{eq:N}
	N(D,x) = \sum_{q \in \Spec\cO_{K,S}} n_q(\sD,x) \log |\kappa(q)|,
\end{equation}
as well as the \defi{truncated counting function}
\begin{equation}
\label{eq:N1}
N^{(1)}(D,x) = \sum_{\substack{q \in \Spec\cO_{K,S} \\ n_q(\sD,x)>0}} 
\log |\kappa(q)|.
\end{equation}
The quantities on the right hand sides of~\eqref{eq:N} and~\eqref{eq:N1} depend on the model $(\sX,\sD)$ and the finite set $S$ only up to functions bounded on $X(\cO_{K,S})$. Since we are interested in these quantities only up to such functions, the notation $N(D,x)$ and $N^{(1)}(D,x)$ does not reflect the model $(\sX,\sD)$ or the finite set $S$.


\subsection{Vojta's conjecture for integral points}

 For a smooth proper Deligne--Mumford stack $\sX \to \Spec \cO_{K,S}$ over the ring of $S$-integers $\cO_{K,S}$ of a number field $K$, we write $X = \sX_K$ for the generic fiber, which we assume is irreducible, and $\underline{X}$ for the coarse moduli space of $X$. Similarly, for a normal crossings divisor $\sD$ of $\sX$, we write $D$ for its generic fiber.

For a divisor $H$ on $\underline{X}$, we denote by $h_H(x)$ the Weil height of $x$ with respect to $H$, which is well-defined up to a bounded function on $\underline{X}(\overline{K})$.  If $H$ is only a divisor on $X$, then some positive integer multiple $rH$ descends to $\underline{X}$. Given a point $x \in X(\overline{K})$ we define $h_H(x) = \frac{1}{r} h_{rH} (\underline{x})$, where $\underline{x}$ is the image of $x$ in $\underline{X}(\overline{K})$.
 
The following is a version of Vojta's conjecture for stacks, applied to integral points: 

\begin{conjecture}
\label{conj:Vojta}
Let $\sX\to \Spec \cO_{K,S}$, $X$, $\underline{X}$, and $D$ be as above. Suppose that $\underline{X}$ is projective, and let $H$ be a big line bundle on it. Fix $\delta>0$. Then there is a proper Zariski-closed subset $Z \subset X$ containing $D$ such that 
\[
	N^{(1)}(D,x) \geq h_{K_X(D)}(x)- \delta h_H(x) - O(1)
\]
for all $x\in \sX(\cO_{K,S})\smallsetminus Z(K)$.
\end{conjecture}

In \cite{AV-Campana-Vojta} we showed that a stronger  conjecture, which applies to points $x\in X(\overline K)$ with $[K(x):K]$ bounded, follows from Vojta's original conjecture for schemes \cite{VojtaABC}. Here we have stated only its outcome for integral points.

\subsection{From Vojta's conjecture to Conjecture~\ref{conj:Campana}}

\begin{lemma}
	\label{lem:Ntoh}
	If $x \in \sX(\cO_{K,S})$ is an $\vec\epsilon$-Campana point, then
	\[
		N^{(1)}(D,x) \leq N(D_{\vec\epsilon},x) \leq h_{D_{\vec\epsilon}}(x) + O(1).
	\]
\end{lemma}

\begin{proof}
	If $x \in \sX(\cO_{K,S})_{\sD_{\vec\epsilon}}$ is an $\vec\epsilon$-Campana point, and $n_q(\sD,x) > 0$, then we have 
	\[
		\sum_i\epsilon_i\cdot n_q(\sD_i,x) = n_q(\sD_{\vec\epsilon},x) \geq 1.
	\] 
The definition~\eqref{eq:N1} of $N^{(1)}(D,x)$ gives
	\begin{align*}
		N^{(1)}(D,x) &= \sum_{n_q(\sD,x) > 0} \log |\kappa(q)|\\
		&\leq\sum_{n_q(\sD,x) > 0} n_q(\sD_{\vec\epsilon},x) \log |\kappa(q)|\\
		&= N(D_{\vec\epsilon},x) \\
		&\leq h_{D_{\vec\epsilon}}(x) + O(1),
	\end{align*}
	where the last inequality follows as in~\cite[p.~1113]{VojtaABC} or~\cite[Theorem~B.8.1(e)]{HindrySilverman}.
\end{proof}

\begin{corollary}
	\label{cor:notdense}
	{\sc Vojta}'s conjecture~\ref{conj:Vojta} implies the $\vec\epsilon$-Campana conjecture~\ref{conj:Campana}.
\end{corollary}

\begin{proof}
Since the divisor $K_X + \sum_i(1 - \epsilon_i)D_i$ is big, we may choose an ample $\QQ$-divisor $H$ such that $K_X + \sum_i(1 - \epsilon_i)D_i - H$ is effective. Let $x \in \sX(\cO_{K,S})_{\sD_{\vec\epsilon}}$ be an $\vec\epsilon$-Campana point. Applying Conjecture~\ref{conj:Vojta}, we obtain a proper Zariski closed subset $\sZ \subset \sX$ such that if $x \notin \sZ$, the inequality
	\[
		N^{(1)}(D,x) \geq h_{K_X(D)}(x)- \delta h_H(x) - O(1)
	\]
	holds. By Lemma~\ref{lem:Ntoh}, we may replace the left hand side with $h_{D_{\vec\epsilon}}(x) + O(1)$ to get
	\[
		h_{D_{\vec\epsilon}}(x) + O(1) \geq h_{K_X(D)}(x)- \delta h_H(x) - O(1).
	\]
	This implies that
	\[
		O(1) \geq h_{K_X\left(\sum_i(1 - \epsilon_i)D_i\right)}(x)- \delta h_H(x).
	\]
	By our choice of $H$, we have
	\[
		O(1) \geq (1 - \delta)h_{K_X\left(\sum_i(1 - \epsilon_i)D_i\right)}(x).
	\]
	Since $K_X + \sum_i(1 - \epsilon_i)D_i$ is big, the set of $x \in \sX(\cO_{K,S})_{\sD_{\vec\epsilon}}$ that avoid $\sZ$ is finite~\cite[Theorem~B.3.2(g)]{HindrySilverman}. Since $\sZ$ is a proper Zariski-closed set, we conclude that the set $\sX(\cO_{K,S})_{\sD_{\vec\epsilon}}$ of $\vec\epsilon$-Campana points is not Zariski-dense in $X$.
\end{proof}

\bibliographystyle{plain}             
\bibliography{levels}

\end{document}